\documentclass[draft]{amsart}

\usepackage{amssymb}
\usepackage{url} 
\usepackage{graphicx}
\usepackage[all]{xy}
\usepackage[all]{xy}

\newtheorem{thm}{Theorem}[section]
\newtheorem{cor}[thm]{Corollary}
\newtheorem{prop}[thm]{Proposition}
\newtheorem{lem}[thm]{Lemma}

\theoremstyle{remark}
\newtheorem{rem}[thm]{Remark}

\theoremstyle{definition}
\newtheorem{dfn}[thm]{Definition}
\newtheorem{ex}[thm]{Example}

\newtheorem{fact}[thm]{Fact}
\newtheorem{const}[thm]{Construction}

\numberwithin{equation}{section}
\numberwithin{thm}{section}

\begin{document}

\title[General heart construction on a triangulated category (I)]{General heart construction on a triangulated category (I): unifying $t$-structures and cluster tilting subcategories}

\author{Hiroyuki NAKAOKA}
\address{Graduate School of Mathematical Sciences, The University of Tokyo 
3-8-1 Komaba, Meguro, Tokyo, 153-8914 Japan}
\email{deutsche@ms.u-tokyo.ac.jp}

\thanks{The author wishes to thank Professor Toshiyuki Katsura for his encouragement}
\thanks{The author wishes to thank his colleague Professor Noriyuki Abe. This work was never possible without his advices.}
\thanks{The author wishes to thank Professor Bernhard Keller and Professor Osamu Iyama for their useful comments and advices, especially on the terminology.}

\begin{abstract}
In the paper of Keller and Reiten, it was shown that the quotient of a triangulated category (with some conditions) by a cluster tilting subcategory becomes an abelian category.
After that, Koenig and Zhu showed in detail, how the abelian structure is given on this quotient category, in a more abstract setting.
On the other hand, as is well known since 1980s, the heart of any $t$-structure is abelian.
We unify these two construction by using the notion of a cotorsion pair.
To any cotorsion pair in a triangulated category, we can naturally associate an abelian category, which gives back each of the above two abelian categories, when the cotorsion pair comes from a cluster tilting subcategory, or is a $t$-structure, respectively.
\end{abstract}

\maketitle

\section{Introduction}
Throughout this paper, we fix a triangulated category $\mathcal{C}$.

For any category $\mathcal{K}$, we write abbreviately $K\in\mathcal{K}$, to indicate that $K$ is an object of $\mathcal{K}$.

For any $K,L\in\mathcal{K}$, let $\mathcal{K}(K,L)$ denote the set of morphisms from $K$ to $L$.

If $\mathcal{M},\mathcal{N}$ are full subcategories of $\mathcal{K}$, then $\mathcal{K}(\mathcal{M},\mathcal{N})=0$ means that $\mathcal{K}(M,N)=0$ for any $M\in\mathcal{M}$ and $N\in\mathcal{N}$.

Similarly, $\mathcal{K}(K,\mathcal{N})=0$ means $\mathcal{K}(K,N)=0$ for any $N\in\mathcal{N}$.

When $\mathcal{K}$ admits $\mathrm{Ext}^{\ell}$, similarly for $\mathrm{Ext}^{\ell}(\mathcal{M},\mathcal{N})=0$ and $\mathrm{Ext}^{\ell}(K,\mathcal{N})=0$ for any integer $\ell$.

\vspace{0.2cm}

As is well known, if $(\mathcal{T}^{\le0},\mathcal{T}^{\ge0})$ is a $t$-structure on $\mathcal{C}$, then its heart $\mathcal{H}=\mathcal{T}^{\le0}\cap\mathcal{T}^{\ge0}$ becomes an abelian category \cite{BBD}. Put $\mathcal{T}^{\le n}:=\mathcal{T}^{\le0}[-n],\mathcal{T}^{\ge n}:=\mathcal{T}^{\ge0}[-n]$ for an integer $n$.
By definition, a $t$-structure is a pair of full additive thick subcategories $(\mathcal{T}^{\le0},\mathcal{T}^{\ge0})$ of $\mathcal{C}$ satisfying

$(t\text{-}1)$ $\mathcal{C}(\mathcal{T}^{\le0},\mathcal{T}^{\ge1})=0$, 

$(t\text{-}2)$ $\mathcal{T}^{\le0}\subseteq\mathcal{T}^{\le1}$ and $\mathcal{T}^{\ge1}\subseteq\mathcal{T}^{\ge0}$,

$(t\text{-}3)$ For any $C\in\mathcal{C}$, there exists a distinguished triangle
\[ X\rightarrow C\rightarrow Y\rightarrow X[1]\quad(X\in\mathcal{T}^{\le0},Y\in\mathcal{T}^{\ge1}). \]

On the other hand, in \cite{K-Z}, Koenig and Zhu showed that for any cluster tilting subcategory $\mathcal{T}$ of $\mathcal{C}$, the quotient category $\mathcal{C}/\mathcal{T}$ carries a naturally induced abelian structure. 
Originally, an equivalence between the quotient of a triangulated category and a certain module category was shown in \cite{B-M-R} for the quotient by a full additive thick subcategory associated to a tilting object in $\mathcal{C}$, and then in \cite{K-R}, for the quotients by a cluster tilting subcategory.

By definition (\cite{K-R}, \cite{K-Z}), a full additive thick subcategory $\mathcal{T}$ of $\mathcal{C}$ is a cluster tilting subcategory if it satisfies

$(s\text{-}1)$ $\mathcal{T}$ is functorially finite (cf.\,\cite{B-R}),

$(s\text{-}2)$ $C\in\mathcal{T}$ if and only if $\mathrm{Ext}^1(C,\mathcal{T})=0$,

$(s\text{-}3)$ $C\in\mathcal{T}$ if and only if $\mathrm{Ext}^1(\mathcal{T},C)=0$.

By $(s\text{-}1)$, it can be easily shown that a cluster tilting subcategory also satisfies the following for any $C\in\mathcal{C}$:\\
$(s\text{-}1)^{\prime}$ For any $C\in\mathcal{C}$, there exists a distinguished triangle
\[ T\rightarrow C\rightarrow T^{\prime}[1]\rightarrow T[1]\quad(T,T^{\prime}\in\mathcal{T}). \]

To unify these two, we introduce the notion of a {\it cotorsion pair} (Definition \ref{DefDivide}). To any cotorsion pair $(\mathcal{U},\mathcal{V})$ in $\mathcal{C}$, we can naturally associate a subfactor category $\underline{\mathcal{H}}:=(\underline{\mathcal{C}}^+\cap\underline{\mathcal{C}}^-)/(\mathcal{U}\cap\mathcal{V})$.
As a main theorem, we show $\underline{\mathcal{H}}$ carries an induced abelian structure (Theorem \ref{ThmHeartAbel}).
In fact, this construction generalizes the above abelian categories, in the following sense (Proposition \ref{PropDivide} and Example \ref{ExHeart}):\\
\-- If $(\mathcal{U},\mathcal{V})$ satisfies $\mathcal{V}\subseteq\mathcal{V}[1]$, then $(\mathcal{U}[-1],\mathcal{V}[1])$ becomes a $t$-structure, and $\underline{\mathcal{H}}$ agrees with the heart of $(\mathcal{U}[-1],\mathcal{V})$.\\
\-- If $(\mathcal{U},\mathcal{V})$ satisfies $\mathcal{U}=\mathcal{V}$, then $\mathcal{T}:=\mathcal{U}=\mathcal{V}$ becomes a cluster tilting subcategory of $\mathcal{C}$, and $\underline{\mathcal{H}}$ agrees with $\mathcal{C}/\mathcal{T}$.
\section{Preliminaries}

\begin{dfn}\label{DefDivide}
Let $\mathcal{U}$ and $\mathcal{V}$ be full additive thick subcategories of $\mathcal{C}$.
We call $(\mathcal{U},\mathcal{V})$ a {\it cotorsion pair} if it satisfies the following.
\begin{enumerate}
\item $\mathrm{Ext}^1(\mathcal{U},\mathcal{V})=0$,
\item For any $C\in\mathcal{C}$, there exists a (not necessarily unique) distinguished triangle
\[ U\rightarrow C\rightarrow V[1]\rightarrow U[1] \]
satisfying $U\in\mathcal{U}$ and $V\in\mathcal{V}$.
\end{enumerate}
For any cotorsion pair $(\mathcal{U},\mathcal{V})$, put
\[ \mathcal{W}:=\mathcal{U}\cap\mathcal{V}. \]
\end{dfn}

\begin{rem} A pair $(\mathcal{U},\mathcal{V})$ of full additive thick subcategories of $\mathcal{C}$ is a cotorsion pair if and only if $(\mathcal{U}[-1],\mathcal{V})$ is a {\it torsion theory} (or a {\it torsion pair}) in \cite{I-Y}. (Unlike \cite{B-R}, it does not require the shift-closedness.)
In this sense, a cotorsion pair is nothing other than a torsion theory with $\mathcal{U}$ shifted by $-1$, and thus not a new notion. However we prefer the above definition, just for the sake of the duality in the index of shifts.
\end{rem}

\begin{rem} A pair $(\mathcal{U},\mathcal{V})$ of full additive subcategories of $\mathcal{C}$ is a cotorsion pair if and only if it satisfies the following conditions for any $C\in\mathcal{C}$.
\begin{enumerate}
\item $C$ belongs to $\mathcal{U}$ if and only if $\mathrm{Ext}^1(C,\mathcal{V})=0$,
\item $C$ belongs to $\mathcal{V}$ if and only if $\mathrm{Ext}^1(\mathcal{U},C)=0$,
\item For any $C\in\mathcal{C}$, there exists a (not necessarily unique) distinguished triangle
\[ U\rightarrow C\rightarrow V[1]\rightarrow U[1] \]
satisfying $U\in\mathcal{U}$ and $V\in\mathcal{V}$.
\end{enumerate}
\end{rem}

\begin{rem}\label{RemDivide}
Let $(\mathcal{U},\mathcal{V})$ be a cotorsion pair in $\mathcal{C}$.
\begin{enumerate}
\item For any $C\in\mathcal{C}$ and any $n\in\mathbb{Z}$, $C$ admits a distinguished triangle
\[ U[n]\rightarrow C\rightarrow V[n+1]\rightarrow U[n+1] \]
with $U\in\mathcal{U},V\in\mathcal{V}$.

\item For any $n\in\mathbb{Z}$, each $\mathcal{U}[n]$ and $\mathcal{V}[n]$ is closed under extensions and direct summands.

\item $\mathcal{W}$ becomes an additive full subcategory of $\mathcal{C}$, closed under direct sums and direct summands, satisfying
\[ \mathrm{Ext}^1(\mathcal{W},\mathcal{W})=0. \]
\end{enumerate}
\end{rem}

\begin{ex}\label{ExDivide}
$\ \ $
\begin{enumerate}
\item If $(\mathcal{T}^{\le0},\mathcal{T}^{\ge0})$ is a $t$-structure on $\mathcal{C}$, then
\[ (\mathcal{U},\mathcal{V}):=(\mathcal{T}^{\le-1},\mathcal{T}^{\ge1}) \]
becomes a cotorsion pair. In this case, $\mathcal{W}=0$.

\item If $\mathcal{T}$ is a cluster tilting subcategory of $\mathcal{C}$, then
\[ (\mathcal{U},\mathcal{V}):=(\mathcal{T},\mathcal{T}) \]
becomes a cotorsion pair. In this case, $\mathcal{W}=\mathcal{T}$.
\end{enumerate}
\end{ex}

Remark that $\mathrm{Ext}^1(\mathcal{U},\mathcal{V})=0$ implies $\mathcal{U}\cap\mathcal{V}[1]=0$. The following proposition shows that $t$-structures and cluster tilting subcategories are characterized as two extremal examples of cotorsion pairs.

\begin{prop}\label{PropDivide}
Let $(\mathcal{U},\mathcal{V})$ be a cotorsion pair.
\begin{enumerate}
\item $\mathcal{V}\subseteq \mathcal{V}[1]$ if and only if $(\mathcal{U},\mathcal{V})=(\mathcal{T}^{\le-1},\mathcal{T}^{\ge1})$ for a $t$-structure $(\mathcal{T}^{\le0},\mathcal{T}^{\ge0})$.

\item $\mathcal{V}=\mathcal{U}$ if and only if $(\mathcal{U},\mathcal{V})=(\mathcal{T},\mathcal{T})$ for a cluster tilting subcategory $\mathcal{T}$.
\end{enumerate}
\end{prop}
\begin{proof}
{\rm (1)} This immediately follows from the definition of a $t$-structure.

{\rm (2)} The \lq only if' part is trivial. To show the converse, put $\mathcal{T}:=\mathcal{U}=\mathcal{V}$.
As in \cite{K-Z} (Definition 3.1 and Lemma 3.2), it suffices to show $\mathcal{T}$ is contravariantly finite.
But this immediately follows from the fact that any object $C\in\mathcal{C}$ admits a distinguished triangle
\[ T\rightarrow C\rightarrow T^{\prime}[1]\rightarrow T[1] \]
for some $T\in\mathcal{U}=\mathcal{T}$ and $T^{\prime}\in\mathcal{V}=\mathcal{T}.$
\end{proof}

\begin{dfn}(cf. \S\S 2.1 in \cite{K-Z})
Let $(\mathcal{U},\mathcal{V})$ be a cotorsion pair in $\mathcal{C}$, and $\mathcal{W}:=\mathcal{U}\cap\mathcal{V}$.
We denote the quotient of $\mathcal{C}$ by $\mathcal{W}$ as
\[ \underline{\mathcal{C}}:=\mathcal{C}/\mathcal{W}. \]

Namely, $\mathrm{Ob}(\underline{\mathcal{C}})=\mathrm{Ob}(\mathcal{C})$, and for any $A,B\in\mathcal{C}$,
\[ \underline{\mathcal{C}}(A,B):=\mathcal{C}(A,B)/\{ f\in\mathcal{C}(A,B)\mid f\ \text{factors through some}\ W\in\mathcal{W} \}. \]

For any morphism $f\in\mathcal{C}(A,B)$, we denote its image in $\underline{\mathcal{C}}(A,B)$ by $\underline{f}$.
This defines an additive functor
\[ \underline{(\ \ )}\colon\mathcal{C}\rightarrow \underline{\mathcal{C}}\ . \]

Since $\mathcal{U}\supseteq\mathcal{W}$ and $\mathcal{V}\supseteq\mathcal{W}$, we also have additive full subcategories of $\underline{\mathcal{C}}$
\[ \underline{\mathcal{U}}:=\mathcal{U}/\mathcal{W}\ \ \text{and}\ \ \underline{\mathcal{V}}:=\mathcal{V}/\mathcal{W}. \]
\end{dfn}

\begin{rem}\label{RemQuotW}
Since $\mathcal{W}$ is closed under direct summands, for any $C\in\mathcal{C}$ we have
\[ C\cong 0\ \ \text{in}\ \  \underline{\mathcal{C}}\ \ \Longleftrightarrow\ \ C\in\mathcal{W}. \]
\end{rem}

Since $\mathrm{Ext}^1(\mathcal{W},\mathcal{W})=0$, Theorem 2.3 in \cite{K-Z} can be applied. Compare with Proposition \ref{PropEpim}.
\begin{fact}\label{FactFact}(Theorem 2.3 in \cite{K-Z})
Let $\mathcal{C},(\mathcal{U},\mathcal{V})$ and $\mathcal{W}$ be as above.
For any distinguished triangle
\[ X\overset{f}{\longrightarrow}Y\overset{g}{\longrightarrow}Z \overset{h}{\longrightarrow}X[1], \]
the following hold.
\begin{enumerate}
\item $\underline{g}$ is epimorphic in $\underline{\mathcal{C}}$ if and only if $\underline{h}=0$,
\item $\underline{g}$ is monomorphic in $\underline{\mathcal{C}}$ if and only if $\underline{f}=0$.
\end{enumerate}
\end{fact}

Let $(\mathcal{U},\mathcal{V})$ be a cotorsion pair in the following.

\begin{prop}\label{PropVanishLift}
$\ \ $
\begin{enumerate}
\item For any $U\in\mathcal{U},C\in\mathcal{C}$ and any $f\in\mathcal{C}(U[-1],C)$, if $\underline{f}=0$, then $f=0$. Namely, we have
\[ \underline{\mathcal{C}}(U[-1],C)=\mathcal{C}(U[-1],C)=\mathrm{Ext}^1(U,C). \]In particular, $C\in\mathcal{V}$ if and only if $\underline{\mathcal{C}}(U[-1],C)=0$ for any $U\in\mathcal{U}$.
\item Dually, for any $V\in\mathcal{V}$ and any $C\in\mathcal{C}$, we have
\[ \underline{\mathcal{C}}(C,V[1])=\mathcal{C}(C,V[1])=\mathrm{Ext}^1(C,V). \]
In particular, $C\in\mathcal{U}$ if and only if $\underline{\mathcal{C}}(C,V[1])=0$ for any $V\in\mathcal{V}$.
\end{enumerate}
\end{prop}
\begin{proof}
{\rm (1)} By definition, $\underline{f}=0$ if and only if $f$ factors through some $W\in\mathcal{W}$. Since $\mathcal{C}(U[-1],W)\subseteq\mathrm{Ext}^1(\mathcal{U},\mathcal{V})=0$, this implies $f=0$.
\[
\xy
(0,8)*+{W}="0";
(-12,-4)*+{U[-1]}="2";
(12,-4)*+{C}="4";
(0,-6)*+{}="6";
{\ar^{0} "2";"0"};
{\ar "0";"4"};
{\ar_{f} "2";"4"};
{\ar@{}|\circlearrowright "0";"6"};
\endxy
\]
{\rm (2)} is shown dually.
\end{proof}

\begin{lem}\label{LemDisApp}
For any cotorsion pair $(\mathcal{U},\mathcal{V})$, we have
\[ \underline{\mathcal{C}}(\mathcal{U},\mathcal{V})=0. \]
\end{lem}
\begin{proof}
Let $f\in\mathcal{C}(U,V)$ be a morphism, where $U\in\mathcal{U}$ and $V\in\mathcal{V}$. By condition {\rm (2)} in Definition \ref{DefDivide}, we can form a distinguished triangle
\[ V^{\prime}\rightarrow U^{\prime}\overset{u}{\longrightarrow}V\overset{v}{\longrightarrow}V^{\prime}[1] \]
where $U^{\prime}\in\mathcal{U}$ and $V^{\prime}\in\mathcal{V}$. Since $\mathcal{V}$ is extension-closed (Remark \ref{RemDivide}), $U^{\prime}$ satisfies $U^{\prime}\in\mathcal{U}\cap\mathcal{V}=\mathcal{W}$.
Since $v\circ f=0$ by $\mathrm{Ext}^1(U,V^{\prime})=0$, $f$ factors through $U^{\prime}$.
\[
\xy
(-24,0)*+{V^{\prime}}="0";
(-8,0)*+{U^{\prime}}="2";
(8,0)*+{V}="4";
(24,0)*+{V^{\prime}[1]}="6";
(8,14)*+{U}="8";
(-2,10)*+{}="10";
{\ar "0";"2"};
{\ar "2";"4"};
{\ar^{f} "8";"4"};
{\ar_{v} "4";"6"};
{\ar@{-->} "8";"2"};
{\ar@{}|\circlearrowright "4";"10"};
\endxy
\]
Since $U^{\prime}\in\mathcal{W}$, this means $\underline{f}=0$.
\end{proof}

\section{Definition of $\mathcal{C}^+$ and $\mathcal{C}^-$}

\begin{lem}\label{LemUniqueUV}
Let $f\colon A\rightarrow B$ be any morphism in $\mathcal{C}$.
\begin{enumerate}
\item Let
\begin{eqnarray*}
U_A\overset{u_A}{\longrightarrow}A\rightarrow V_A[1]\rightarrow U_A[1]\\
U_B\overset{u_B}{\longrightarrow}B\rightarrow V_B[1]\rightarrow U_B[1]
\end{eqnarray*}
be any distinguished triangles satisfying $U_A,U_B\in\mathcal{U}$ and $V_A,V_B\in\mathcal{V}$.
Then there exists a morphism $f_U\in\mathcal{C}(U_A,U_B)$ such that
\begin{eqnarray*}
f\circ u_A=u_B\circ f_U.\\
\xy
(-8,8)*+{U_A}="0";
(-8,-8)*+{U_B}="2";
(8,8)*+{A}="4";
(8,-8)*+{B}="6";
{\ar_{f_U} "0";"2"};
{\ar_{u_B} "2";"6"};
{\ar^{u_A} "0";"4"};
{\ar^{f} "4";"6"};
{\ar@{}|\circlearrowright "0";"6"};
\endxy
\end{eqnarray*}
Moreover, $f_U$ with this property is unique in $\underline{\mathcal{C}}(U_A,U_B)$.
\item Dually, for any distinguished triangles
\begin{eqnarray*}
U^{\prime}_A[-1]\rightarrow A\rightarrow V^{\prime}_A\rightarrow U^{\prime}_A\\
U^{\prime}_B[-1]\rightarrow B\rightarrow V^{\prime}_B\rightarrow U^{\prime}_B
\end{eqnarray*}
with $U^{\prime}_A,U^{\prime}_B\in\mathcal{U}$ and $V^{\prime}_A,V^{\prime}_B\in\mathcal{V}$, there exists a morphism $f^{\prime}_V\in\mathcal{C}(V^{\prime}_A,V^{\prime}_B)$ compatible with $f$, uniquely up to $\mathcal{W}$.
\end{enumerate}
\end{lem}
\begin{proof}
We only show {\rm (1)}. Existence immediately follows from $\mathcal{C}(U_A,V_B[1])=0$. Moreover if $f_U^1$ and $f_U^2$ in $\mathcal{C}(U_A,U_B)$ satisfies
\[ f_U^1\circ u_A=u_B\circ f_U=f^2_U\circ u_A, \]
then by $(f_U^1-f_U^2)\circ u_A=0$, there exists $w\in\mathcal{C}(U_A,V_B)$ such that $f_U^1-f_U^2$ factors $w$.
\[
\xy
(-24,-8)*+{V_B}="0";
(-8,-8)*+{U_B}="2";
(8,-8)*+{B}="4";
(24,-8)*+{V_B[1]}="6";
(-8,8)*+{U_A}="8";
(8,8)*+{A}="10";
(-18,4)*+{}="12";
{\ar "0";"2"};
{\ar "2";"4"};
{\ar "4";"6"};
{\ar^{f_U^1-f_U^2} "8";"2"};
{\ar^{f} "10";"4"};
{\ar@{-->}_{w} "8";"0"};
{\ar^{u_A} "8";"10"};
{\ar@{}|\circlearrowright "2";"12"};
\endxy
\]
By Lemma \ref{LemDisApp} we have $\underline{w}=0$, and thus $\underline{f_U^1}=\underline{f_U^2}$.
\end{proof}

\begin{prop}\label{PropUniqueUV}
Let $C$ be any object in $\mathcal{C}$.
\begin{enumerate}
\item For any distinguished triangles
\begin{eqnarray*}
U\overset{u}{\longrightarrow}C\rightarrow V[1]\rightarrow U[1]\\
U^{\prime}\overset{u^{\prime}}{\longrightarrow}C\rightarrow V^{\prime}[1]\rightarrow U^{\prime}[1]
\end{eqnarray*}
satisfying $U,U^{\prime}\in\mathcal{U}$ and $V,V^{\prime}\in\mathcal{V}$, there exists a morphism $s\in\mathcal{C}(U,U^{\prime})$ compatible with $u$ and $u^{\prime}$, such that $\underline{s}$ is an isomorphism.
\[
\xy
(0,8)*+{U}="0";
(0,-8)*+{U^{\prime}}="2";
{\ar_{\rotatebox{90}{$\cong$}}^{\underline{s}} "0";"2"};
\endxy
\qquad
\xy
(-4,8)*+{U}="0";
(-4,-8)*+{U^{\prime}}="2";
(6,0)*+{C}="4";
(-8,0)*+{}="6";
{\ar_{s} "0";"2"};
{\ar^{u} "0";"4"};
{\ar_{u^{\prime}} "2";"4"};
{\ar@{}|\circlearrowright "4";"6"};
\endxy
\]
\item Dually, those $V$ appearing in distinguished triangles
\[ U[-1]\rightarrow C\rightarrow V\rightarrow U\quad(U\in\mathcal{U},V\in\mathcal{V}) \]
are isomorphic in $\underline{\mathcal{C}}$.
\end{enumerate}
\end{prop}
\begin{proof}
This immediately follows from Lemma \ref{LemUniqueUV}.
\end{proof}

\begin{cor}\label{CorC+}
For any $C\in\mathcal{C}$, the following are equivalent.
\begin{enumerate}
\item There exists a distinguished triangle 
\[ W_0\rightarrow C\rightarrow V_0[1]\rightarrow W_0[1] \]
such that $W_0\in\mathcal{W},V_0\in\mathcal{V}$.\\
\item Any distinguished triangle 
\[ U\rightarrow C\rightarrow V[1]\rightarrow U[1]\quad (U\in\mathcal{U},V\in\mathcal{V}) \]
satisfies $U\in\mathcal{W}$.
\end{enumerate}
\end{cor}
\begin{proof}
Suppose {\rm (1)} holds. By Proposition \ref{PropUniqueUV}, we have $U\cong W_0$ in $\underline{\mathcal{C}}$. By Remark \ref{RemQuotW}, this means $U\in\mathcal{W}$. The converse is trivial.
\end{proof}

Dually, we have the following:
\begin{cor}\label{CorC-}
For any $C\in\mathcal{C}$, the following are equivalent.
\begin{enumerate}
\item There exists a distinguished triangle 
\[ U_0[-1]\rightarrow C\rightarrow W_0\rightarrow U_0 \]
such that $U_0\in\mathcal{U},W_0\in\mathcal{W}$.\\
\item Any distinguished triangle 
\[ U[-1]\rightarrow C\rightarrow V\rightarrow U \quad (U\in\mathcal{U},V\in\mathcal{V}) \]
satisfies $V\in\mathcal{W}$.
\end{enumerate}
\end{cor}

\begin{dfn}\label{DefC+-}
$\ \ $
\begin{enumerate}
\item $\mathcal{C}^+$ is defined to be the full subcategory of $\mathcal{C}$, consisting of objects satisfying equivalent conditions of Corollary \ref{CorC+}.
\item $\mathcal{C}^-$ is defined to be the full subcategory of $\mathcal{C}$, consisting of objects satisfying equivalent conditions of Corollary \ref{CorC-}.
\end{enumerate}
\end{dfn}

\begin{rem}\label{RemRem} The following are satisfied.
\begin{enumerate}
\item Each of $\mathcal{C}^+$ and $\mathcal{C}^-$ is an additive full subcategory of $\mathcal{C}$ containing $\mathcal{W}$.
\item $\mathcal{C}^+\supseteq\mathcal{V}[1]$.
\item $\mathcal{C}^-\supseteq\mathcal{U}[-1]$.
\end{enumerate}
\end{rem}

\begin{dfn}\label{DefHeart}
For any cotorsion pair $(\mathcal{U},\mathcal{V})$, put $\mathcal{H}:=\mathcal{C}^+\cap\mathcal{C}^-$.
Since $\mathcal{H}\supseteq\mathcal{W}$, we have an additive full subcategory
\[ \underline{\mathcal{H}}:=\mathcal{H}/\mathcal{W}\subseteq\underline{\mathcal{C}}, \]
which we call {\it the heart} of $(\mathcal{U},\mathcal{V})$.
\end{dfn}

\begin{ex}\label{ExHeart}
$\ \ $
\begin{enumerate}
\item If $(\mathcal{U},\mathcal{V})=(\mathcal{T}^{\le -1},\mathcal{T}^{\ge 1})$, where $(\mathcal{T}^{\le0},\mathcal{T}^{\ge0})$ is a $t$-structure, then we have
\begin{eqnarray*}
\mathcal{C}^-=\mathcal{T}^{\le -1}[-1]=\mathcal{T}^{\le0},\\
\mathcal{C}^+=\mathcal{T}^{\ge 1}[1]=\mathcal{T}^{\ge0},\\
\underline{\mathcal{H}}=\mathcal{H}=\mathcal{T}^{\le0}\cap\mathcal{T}^{\ge0}.
\end{eqnarray*}
Thus the definition of the heart agrees with that of a $t$-structure. Thus $\underline{\mathcal{H}}$ is abelian, and admits a cohomological functor $H^0\colon\mathcal{C}\rightarrow \underline{\mathcal{H}}$ (cf. \cite{BBD}).

\item If $\mathcal{U}=\mathcal{V}=\mathcal{T}$ is a cluster tilting subcategory of $\mathcal{C}$, then we have
\begin{eqnarray*}
\mathcal{C}^+=\mathcal{C}^-=\mathcal{H}=\mathcal{C},\\
\underline{\mathcal{H}}=\mathcal{C}/\mathcal{T}.
\end{eqnarray*}
By \cite{K-Z}, $\underline{\mathcal{H}}$ becomes an abelian category, and the quotient functor $\underline{(\ \ )}\colon\mathcal{C}\rightarrow\mathcal{C}/\mathcal{T}=\underline{\mathcal{H}}$ is cohomological.
\end{enumerate}
\end{ex}

\section{Existence of (co)reflections}

Since $\mathcal{C}^+\cap\mathcal{C}^-=\mathcal{H}\supseteq\mathcal{W}$, we have additive full subcategories of $\underline{\mathcal{C}}$
\[ \underline{\mathcal{C}}^+:=\mathcal{C}^+/\mathcal{W}\quad \text{and}\quad \underline{\mathcal{C}}^-:=\mathcal{C}^-/\mathcal{W}. \]
\[
\xy
(-16,0)*+{\underline{\mathcal{H}}}="0";
(0,8)*+{\underline{\mathcal{C}}^+}="2";
(0,-8)*+{\underline{\mathcal{C}}^-}="4";
(16,0)*+{\underline{\mathcal{C}}}="6";
{\ar@{^(->} "0";"2"};
{\ar@{^(->} "0";"4"};
{\ar@{^(->} "2";"6"};
{\ar@{^(->} "4";"6"};
{\ar@{}|\circlearrowright "0";"6"};
\endxy
\]

\begin{dfn}(Definition 3.1.1 in \cite{Borceux})
Let $\mathcal{A}$ and $\mathcal{B}$ be categories, and $F\colon A\rightarrow B$ be a functor. For any $B\in\mathcal{B}$, a reflection of $B$ along $F$ is a pair $(R_B,\eta_B)$ of $R_B\in\mathcal{A}$ and $\eta_B\in\mathcal{B}(B,F(R_B))$, satisfying the following universality:

$\mathrm{(\ast)}$ For any $A\in\mathcal{A}$ and any $b\in\mathcal{B}(B,F(A))$, there exists a unique morphism $a\in\mathcal{A}(R_B, A)$ such that $F(a)\circ\eta_B=b$.
\[
\xy
(0,-8)*+{F(A)}="0";
(-10,6)*+{B}="2";
(10,6)*+{F(R_B)}="4";
(0,8)*+{}="6";
{\ar_{b} "2";"0"};
{\ar^{\eta_B} "2";"4"};
{\ar^{F(a)} "4";"0"};
{\ar@{}|\circlearrowright "0";"6"};
\endxy
\]
A {\it coreflection} is defined dually.
\end{dfn}

\begin{const}\label{ConstReflect}
For any $C\in\mathcal{C}$, consider a diagram
\[
\xy
(-16,16)*+{U_C^{\prime}[-1]}="0";
(-15,-1)*+{U_C}="2";
(-13.5,-17.5)*+{V_C^{\prime}}="4";
(-4,2)*+{C}="6";
(2.3,-5.3)*+{Z_C}="8";
(16.5,8.5)*+{V_C[1]}="10";
(-8.5,-7.5)*+_{_{\circlearrowright}}="12";
(-11.5,5.5)*+_{_{\circlearrowright}}="14";
(4.3,0.5)*+_{_{\circlearrowright}}="14";
{\ar_{u_C^{\prime}} "0";"2"};
{\ar_{v_C^{\prime}} "2";"4"};
{\ar^{w_C} "0";"6"};
{\ar_{u_C} "2";"6"};
{\ar_{z_C} "6";"8"};
{\ar^{v_C} "6";"10"};
{\ar_{x_C} "4";"8"};
{\ar_{y_C} "8";"10"};
\endxy
\]
where
\begin{eqnarray*}
U_C\overset{u_{C}}{\longrightarrow}C\overset{v_C}{\longrightarrow}V_C[1]\rightarrow U_C[1]\\
U_C^{\prime}[-1]\overset{u_{C}^{\prime}}{\longrightarrow}U_C\overset{v_C^{\prime}}{\longrightarrow}V_C^{\prime}\rightarrow U_C^{\prime}\\
U_C^{\prime}[-1]\overset{w_{C}}{\longrightarrow}C\overset{z_C}{\longrightarrow}Z_C\rightarrow U_C^{\prime}
\end{eqnarray*}
are distinguished triangles, satisfying $U_C,U_C^{\prime}\in\mathcal{U}$ and $V_C,V_C^{\prime}\in\mathcal{V}$.
Since $\mathcal{U}$ is extension-closed, we have $V_C^{\prime}\in\mathcal{W}$.
By the octahedron axiom,
\[ V_C^{\prime}\overset{x_C}{\longrightarrow}Z_C\overset{y_C}{\longrightarrow}V_C[1]\rightarrow V_C^{\prime} \]
also becomes a distinguished triangle.
Thus $Z_C$ belongs to $\mathcal{C}^+$.
\end{const}

\begin{prop}\label{PropReflection}
In the notation of Construction \ref{ConstReflect}, for any $C\in\mathcal{C}$,
\[ \underline{z_C}\colon C\rightarrow Z_C \]
gives a reflection of $C$ along $\underline{\mathcal{C}}^+\hookrightarrow\underline{\mathcal{C}}$.
\end{prop}
\begin{proof}
Let $Y$ be any object in $\mathcal{C}^+$, and let $y\in\mathcal{C}(C,Y)$ be any morphism.
It suffices to show that there exists a unique morphism
\[ \underline{q}\in\underline{\mathcal{C}}(Z_C,Y) \]
such that $\underline{q}\circ\underline{z_C}=\underline{y}$.

\[
\xy
(-4.5,8.5)*+{C}="0";
(0,-6)*+{Y}="2";
(11,5)*+{Z_C}="4";
(8,-4)*+{}="6";
{\ar_{\underline{y}} "0";"2"};
{\ar^{\underline{z_C}} "0";"4"};
{\ar^{\underline{q}} "4";"2"};
{\ar@{}|\circlearrowright "0";"6"};
\endxy
\]
In fact, $q$ can be chosen to satisfy $q\circ z_C=y$.

First, we show the existence.
Since
\[ U_C^{\prime}[-1]\overset{w_C}{\longrightarrow}C\overset{z_C}{\longrightarrow}Z_C\rightarrow U_C^{\prime} \]
is a distinguished triangle, it suffices to show $y\circ w_C=0$.

Let
\[ U_Y\overset{u_Y}{\longrightarrow}Y\overset{v_Y}{\longrightarrow}V_Y[1]\rightarrow U_Y[1] \]
be a distinguished triangle such that $U_Y\in\mathcal{W}, V_Y\in\mathcal{V}$.
By $\mathcal{C}(U_C,V_Y[1])=0$, there exists a morphism $y_C\in\mathcal{C}(U_C,U_Y)$ such that $y\circ u_C=u_Y\circ y_U$.
\[
\xy
(0,16)*+{U_C^{\prime}[-1]}="0";
(-10,6)*+{U_C}="2";
(-10,-6)*+{U_Y}="4";
(10,6)*+{C}="6";
(10,-6)*+{Y}="8";
(0,2)*+{}="10";
{\ar_{u_C^{\prime}} "0";"2"};
{\ar_{y_U} "2";"4"};
{\ar_{u_C} "2";"6"};
{\ar^{w_C} "0";"6"};
{\ar^{y} "6";"8"};
{\ar_{u_Y} "4";"8"};
{\ar@{}|\circlearrowright "2";"8"};
{\ar@{}|\circlearrowright "0";"10"};
\endxy
\]

Since $U_Y\in\mathcal{W}$, we have
\[ \underline{y}\circ\underline{w_C}=\underline{u_Y}\circ\underline{y_U}\circ\underline{u_C^{\prime}}=0. \]
By Proposition \ref{PropVanishLift}, we obtain $y\circ w_C=0$.

To show the uniqueness,
suppose $q,q^{\prime}\in\mathcal{C}(Z_C,Y)$ satisfies
\begin{eqnarray*}
\underline{q}\circ\underline{z_C}=\underline{y}=\underline{q^{\prime}}\circ\underline{z_C}.\\
\xy
(-7,13)*+{C}="0";
(10,11)*+{Z_C}="2";
(-16,0)*+{U_Y}="4";
(0,0)*+{Y}="6";
(0,14)*+{}="7";
(16,0)*+{V_Y[1]}="8";
{\ar^{z_C} "0";"2"};
{\ar_{y} "0";"6"};
{\ar_{u_Y} "4";"6"};
{\ar_{v_Y} "6";"8"};
{\ar|*{_{q,q^{\prime}}} "2";"6"};
{\ar@{}|\circlearrowright "6";"7"};
\endxy
\end{eqnarray*}
Since $\underline{v_Y}\circ(\underline{q}-\underline{q^{\prime}})\circ\underline{z_C}=0$ in $\underline{\mathcal{C}}(C,V_Y[1])$, it follows $v_Y\circ(q-q^{\prime})\circ z_C=0$ by Proposition \ref{PropVanishLift}.

Thus we have a morphism of triangles
\[
\xy
(-16,8)*+{C}="0";
(0,8)*+{Z_C}="2";
(16,8)*+{U_C^{\prime}}="4";
(-16,-8)*+{U_Y}="10";
(0,-8)*+{Y}="12";
(16,-8)*+{V_Y[1]\, .}="14";
{\ar^{z_C} "0";"2"};
{\ar "2";"4"};
{\ar "0";"10"};
{\ar "4";"14"};
{\ar|*+{_{q-q^{\prime}}} "2";"12"};
{\ar_{u_Y} "10";"12"};
{\ar_{v_Y} "12";"14"};
{\ar@{}|\circlearrowright "0";"12"};
{\ar@{}|\circlearrowright "2";"14"};
\endxy
\]
Since $\mathcal{C}(U_C^{\prime},V_Y[1])=0$, this implies $v_Y\circ(q-q^{\prime})=0$. Thus $q-q^{\prime}$ factors through $U_Y\in\mathcal{W}$, which means $\underline{q}=\underline{q^{\prime}}$.
\end{proof}

\begin{cor}\label{CorAdjoint}$($Proposition 3.1.2 and Proposition 3.1.3 in \cite{Borceux}$)$
$\ \ $
\begin{enumerate}
\item Since $(Z_C,z_C)$ is a reflection of $C$ along $\underline{\mathcal{C}}^+\hookrightarrow\underline{\mathcal{C}}$, it is determined up to a canonical isomorphism in $\underline{\mathcal{C}}^+$.
\item As in \cite{Borceux}, if we allow the axiom of choice, we obtain a left adjoint $\sigma\colon\underline{\mathcal{C}}\rightarrow\underline{\mathcal{C}}^+$ of the inclusion $\underline{\mathcal{C}}^+\overset{i^+}{\hookrightarrow}\underline{\mathcal{C}}$. If we denote the adjunction by $\eta\colon\mathrm{Id}_{\underline{\mathcal{C}}}\Longrightarrow i^+\circ\sigma$, then there exists a canonical isomorphism $Z_C\cong\sigma(C)$ in $\underline{\mathcal{C}}$, compatible with $\underline{z_C}$ and $\eta_C$.
\[
\xy
(-8,8)*+{C}="0";
(8,8)*+{Z_C}="2";
(0,-5)*+{\sigma(C)}="4";
(0,12)*+{}="6";
{\ar^{\underline{z_C}} "0";"2"};
{\ar_{\eta_C} "0";"4"};
{\ar^{\cong} "2";"4"};
{\ar@{}|\circlearrowright "6";"4"};
\endxy
\]
\end{enumerate}
\end{cor}

Dually, we have the following:
\begin{rem}\label{RemCoReflect}
For any $C\in\mathcal{C}$, if we take a diagram
\[
\xy
(-20,16)*+{U_C^{\prime\prime}[-1]}="0";
(0.5,15)*+{K_C}="2";
(17,14)*+{U_C^{\prime\prime\prime}}="4";
(-2,4)*+{C}="6";
(6.1,-0.6)*+{V_C^{\prime\prime}}="8";
(-5.8,-14.2)*+{V_C^{\prime\prime\prime}[1]}="10";
(7.5,8.5)*+_{_{\circlearrowright}}="12";
(-5.5,11.5)*+_{_{\circlearrowright}}="14";
(-0.5,-4.3)*+_{_{\circlearrowright}}="14";
{\ar^{} "0";"2"};
{\ar^{} "2";"4"};
{\ar_{} "0";"6"};
{\ar^{k_C} "2";"6"};
{\ar^{} "6";"8"};
{\ar_{} "6";"10"};
{\ar^{} "4";"8"};
{\ar^{} "8";"10"};
\endxy
\quad (U_C^{\prime\prime},U_C^{\prime\prime\prime}\in\mathcal{U},V_C^{\prime\prime},V_C^{\prime\prime\prime}\in\mathcal{V})
\]
where all $180^\circ$ composition of arrows are distinguished triangles, then $\underline{k_C}\colon K_C\rightarrow C$ gives a coreflection of $C$ along $\underline{\mathcal{C}}^-\hookrightarrow\underline{\mathcal{C}}$.
Thus $K_C$ is uniquely determined up to a canonical isomorphism in $\underline{\mathcal{C}}^-$. Moreover if we allow the axiom of choice, we obtain a right adjoint of $\underline{\mathcal{C}}^-\hookrightarrow\underline{\mathcal{C}}$.
\end{rem}

\begin{lem}\label{LemHeart}
In the notation of Construction \ref{ConstReflect}, if $C$ belongs to $\mathcal{C}^-$, then $Z_C$ belongs to $\mathcal{H}$.

Dually, $C\in\mathcal{C}^+$ implies $K_C\in\mathcal{H}$ $($in the notation of Remark \ref{RemCoReflect}$)$.
\end{lem}
\begin{proof}
We only show the former half. The latter is shown dually.
Let
\begin{eqnarray*}
U[-1]\overset{u}{\longrightarrow}C\overset{v}{\longrightarrow}V\rightarrow U\\
U_Z[-1]\overset{u_Z}{\longrightarrow}Z_C\overset{v_Z}{\longrightarrow}V_Z\rightarrow U_Z
\end{eqnarray*}
be distinguished triangles satisfying $U,U_Z\in\mathcal{U}$ and $V,V_Z\in\mathcal{V}$.
By assumption, $V\in\mathcal{W}$. It suffices to show $V_Z\in\mathcal{W}$.

In the notation of Construction \ref{ConstReflect}, complete $z_C\circ u\colon U[-1]\rightarrow Z_C$ into a distinguished triangle
\[ U[-1]\overset{z_C\circ u}{\longrightarrow}Z_C\rightarrow Q\rightarrow U. \]
By the octahedron axiom, we also have a distinguished triangle
\[ V\rightarrow Q\rightarrow U_C^{\prime}\rightarrow V[1], \]
and thus $Q\in\mathcal{U}$.
\[
\xy
(-20,16)*+{U[-1]}="0";
(0.5,15)*+{C}="2";
(17,14)*+{V}="4";
(-2,4)*+{Z_C}="6";
(6.3,-0.8)*+{Q}="8";
(-5.8,-14.2)*+{U_C^{\prime}}="10";
(7.5,8.5)*+_{_{\circlearrowright}}="12";
(-5.5,11.5)*+_{_{\circlearrowright}}="14";
(-0.5,-4.3)*+_{_{\circlearrowright}}="14";
{\ar^{u} "0";"2"};
{\ar^{} "2";"4"};
{\ar_{z_C\circ u} "0";"6"};
{\ar^{z_C} "2";"6"};
{\ar^{} "6";"8"};
{\ar_{} "6";"10"};
{\ar^{} "4";"8"};
{\ar^{} "8";"10"};
\endxy
\]
Since $\mathcal{C}(U[-1],V_Z)=0$, we obtain a morphism of triangles
\[
\xy
(-24,8)*+{U[-1]}="0";
(-8,8)*+{Z_C}="2";
(8,8)*+{Q}="4";
(24,8)*+{U}="6";
(-24,-8)*+{U_Z[-1]}="10";
(-8,-8)*+{Z_C}="12";
(8,-8)*+{V_Z}="14";
(24,-8)*+{U_Z.}="16";
{\ar "0";"10"};
{\ar@{=} "2";"12"};
{\ar "4";"14"};
{\ar "6";"16"};
{\ar "0";"2"};
{\ar "10";"12"};
{\ar "2";"4"};
{\ar_{v_Z} "12";"14"};
{\ar "4";"6"};
{\ar "14";"16"};
{\ar@{}|\circlearrowright "0";"12"};
{\ar@{}|\circlearrowright "2";"14"};
{\ar@{}|\circlearrowright "4";"16"};
\endxy
\]
Since $\underline{\mathcal{C}}(Q,V_Z)=0$ by Lemma \ref{LemDisApp}, it follows $\underline{v_Z}=0$.
For any $V^{\dag}\in\mathcal{V}$ and any $v^{\dag}\in\mathcal{C}(V_Z,V^{\dag}[1])$ we have $\underline{v^{\dag}\circ v_Z}=\underline{v^{\dag}}\circ 0=0$ in $\underline{\mathcal{C}}(Z_C,V^{\dag}[1])$.
By Proposition \ref{PropVanishLift}, we obtain $v^{\dag}\circ v_Z=0$.
Thus $v^{\dag}$ factors through $U_Z$, which means $v^{\dag}=0$ since $\mathcal{C}(U_Z, V^{\dag}[1])=0$.
\[
\xy
(-8,0)*+{Z_C}="0";
(8,0)*+{V_Z}="2";
(24,0)*+{U_Z}="4";
(8,-14)*+{V^{\dag}[1]}="6";
(18,-10)*+{}="10";
{\ar^{v_Z} "0";"2"};
{\ar "2";"4"};
{\ar_{v^{\dag}} "2";"6"};
{\ar^{0} "4";"6"};
{\ar@{}|\circlearrowright "2";"10"};
\endxy
\]
Thus we have $\mathcal{C}(V_Z,\mathcal{V}[1])=0$, namely $V_Z\in\mathcal{W}$.
\end{proof}

\section{Existence of (co-)kernels}

\begin{lem}\label{LemConstofM}
Let $A\overset{f}{\longrightarrow}B$ be any morphism in $\mathcal{C}$.
Take a diagram
\[
\xy
(0,12)*+{U_A[-1]}="0";
(0,0)*+{A}="2";
(0,-12)*+{V_A}="4";
(16,0)*+{B}="6";
(32,-12)*+{M_f}="8";
(10,8)*+{}="10";
{\ar_{u_A} "0";"2"};
{\ar_{v_A} "2";"4"};
{\ar^{} "0";"6"};
{\ar_{f} "2";"6"};
{\ar^{m_f} "6";"8"};
{\ar@{}|\circlearrowright "2";"10"};
\endxy
\]
where
\begin{eqnarray*}
U_A[-1]\overset{u_A}{\longrightarrow}A\overset{v_A}{\longrightarrow}V_A\rightarrow U_A\\
U_A[-1]\rightarrow B\overset{m_f}{\longrightarrow}M_f\rightarrow U_A
\end{eqnarray*}
are distinguished triangles, satisfying $U_A\in\mathcal{U}_A,V_A\in\mathcal{V}$. Then we have the following.
\begin{enumerate}
\item $A\in\mathcal{C}^-\ \Longrightarrow\ \underline{m_f}\circ\underline{f}=0$,
\item $B\in\mathcal{C}^-\ \Longrightarrow\ M_f\in\mathcal{C}^-$.
\end{enumerate}
\end{lem}
\begin{proof}
{\rm (1)} This immediately follows from $V_A\in\mathcal{W}$.

{\rm (2)} Take distinguished triangles
\begin{eqnarray*}
U_B[-1]\rightarrow B\rightarrow V_B\rightarrow U_B\\
U_M[-1]\rightarrow M_f\rightarrow V_M\rightarrow U_M\\
(U_B,U_M\in\mathcal{U},\ \ V_B,V_M\in\mathcal{V}).
\end{eqnarray*}
By assumption, $V_B\in\mathcal{W}$. It suffices to show $V_M\in\mathcal{W}$.
Since $\mathcal{C}(U_B[-1],V_M)=0$, there exists a morphism of triangles
\[
\xy
(-28,8)*+{U_B[-1]}="0";
(-8,8)*+{B}="2";
(8,8)*+{V_B}="4";
(24,8)*+{U_B}="6";
(-28,-8)*+{U_M[-1]}="10";
(-8,-8)*+{M_f}="12";
(8,-8)*+{V_M}="14";
(24,-8)*+{U_M}="16";
{\ar_{m_U} "0";"10"};
{\ar_{m_f} "2";"12"};
{\ar_{m_V} "4";"14"};
{\ar "6";"16"};
{\ar^{u_B} "0";"2"};
{\ar_{u_M} "10";"12"};
{\ar^{v_B} "2";"4"};
{\ar_{v_M} "12";"14"};
{\ar "4";"6"};
{\ar "14";"16"};
{\ar@{}|\circlearrowright "0";"12"};
{\ar@{}|\circlearrowright "2";"14"};
{\ar@{}|\circlearrowright "4";"16"};
\endxy
\]
Let $V^{\dag}$ be any object in $\mathcal{V}$, and $v^{\dag}\in\mathcal{C}(V_M,V^{\dag}[1])$ be any morphism. It suffices to show $v^{\dag}=0$. Since $v^{\dag}\circ m_V\in\mathcal{C}(V_B,V^{\dag}[1])=0$, we have $v^{\dag}\circ v_M\circ m_f=0$.

\[
\xy
(-8,14)*+{B}="0";
(-8,0)*+{M_f}="2";
(-8,-14)*+{U_A}="4";
(8,14)*+{V_B}="10";
(15,14)*+{\in\mathcal{W}}="11";
(8,0)*+{V_M}="12";
(8,-14)*+{V^{\dag}[1]}="14";
(24,0)*+{U_M}="16";
(18,-10)*+{}="18";
{\ar^{v_B} "0";"10"};
{\ar_{v_M} "2";"12"};
{\ar@{-->}_{0} "4";"14"};
{\ar_{m_f} "0";"2"};
{\ar "2";"4"};
{\ar^{m_V} "10";"12"};
{\ar_{v^{\dag}} "12";"14"};
{\ar "12";"16"};
{\ar@{-->}^{0} "16";"14"};
{\ar@{}|\circlearrowright "0";"12"};
{\ar@{}|\circlearrowright "2";"14"};
{\ar@{}|\circlearrowright "18";"12"};
\endxy
\]
Thus $v^{\dag}\circ v_M$ factors through $U_A$, which means $v^{\dag}\circ v_M=0$, since $\mathcal{C}(U_A,V^{\dag}[1])=0$.
Thus $v^{\dag}$ factors through $U_M$, and $v^{\dag}=0$ follows from $\mathcal{C}(U_M,V^{\dag}[1])=0$.
\end{proof}

Dually, we have the following:
\begin{rem}\label{RemConstofL}
For any morphism $A\overset{f}{\longrightarrow}B$ in $\mathcal{C}$, consider a diagram
\[
\xy
(0,-12)*+{V_B[1]}="0";
(0,0)*+{B}="2";
(0,12)*+{U_B}="4";
(-16,0)*+{A}="6";
(-32,12)*+{L_f}="8";
(-10,-8)*+{}="10";
{\ar^{v_B} "2";"0"};
{\ar_{u_B} "4";"2"};
{\ar_{v_B\circ f} "6";"0"};
{\ar^{f} "6";"2"};
{\ar^{\ell_f} "8";"6"};
{\ar@{}|\circlearrowright "2";"10"};
\endxy
\]
where
\begin{eqnarray*}
U_B\overset{u_B}{\longrightarrow} B\overset{v_B}{\longrightarrow} V_B[1]\rightarrow U_B[1]\\
V_B\rightarrow L_f\overset{\ell_f}{\longrightarrow} A\rightarrow V_B[1]
\end{eqnarray*}
are distinguished triangles satisfying $U_B\in\mathcal{U}, V_B\in\mathcal{V}$.
Then, $A\in\mathcal{C}^+$ implies $L_f\in\mathcal{C}^+$, and $B\in\mathcal{C}^+$ implies $\underline{f}\circ\underline{\ell_f}$=0.
\end{rem}

\begin{prop}\label{PropUnivofM}
Let $A\overset{f}{\longrightarrow}B$ be any morphism in $\mathcal{C}$.
Then $m_f\colon B\rightarrow M_f$ in Lemma \ref{LemConstofM} satisfies the following property$:$

$(\ast)$\ \ For any $C\in\mathcal{C}$ and any morphism $g\in\mathcal{C}(B,C)$ satisfying $\underline{g}\circ\underline{f}=0$, there exists a morphism $c\in\mathcal{C}(M_f,C)$ such that $c\circ m_f=g$.
\[
\xy
(-8,0)*+{A}="0";
(10,0)*+{B}="2";
(30,0)*+{C}="4";
(20,-12)*+{M_f}="6";
(20,2)*+{}="10";
{\ar^{f} "0";"2"};
{\ar^{g} "2";"4"};
{\ar_{m_f} "2";"6"};
{\ar_{c} "6";"4"};
{\ar@{}|\circlearrowright "6";"10"};
\endxy
\]
Moreover if $C\in\mathcal{C}^+$, then $\underline{c}\in\underline{\mathcal{C}}(M_f,C)$ satisfying
\[ \underline{c}\circ\underline{m_f}=\underline{g} \]
is unique in $\underline{\mathcal{C}}(M_f,C)$.
The dual statement also holds for $L_f$ in Remark \ref{RemConstofL}.
\end{prop}
\begin{proof}
First we show the existence.
By Proposition \ref{PropVanishLift}, $\underline{g}\circ\underline{f}\circ\underline{u_A}=0$ means $g\circ f\circ u_A=0$.
Thus there exists $c\in\mathcal{C}(M_f,C)$ such that $c\circ m_f=g$.
\[
\xy
(-16,14)*+{U_A[-1]}="0";
(-16,0)*+{A}="2";
(0,0)*+{B}="4";
(16,6)*+{C}="6";
(12,-11)*+{M_f}="8";
(20,-4)*+{}="10";
(-7,9)*+{}="12";
{\ar_{u_A} "0";"2"};
{\ar_{f} "2";"4"};
{\ar^{f\circ u_A} "0";"4"};
{\ar^{g} "4";"6"};
{\ar_{m_f} "4";"8"};
{\ar@{-->}_{c} "8";"6"};
{\ar@{}|\circlearrowright "4";"10"};
{\ar@{}|\circlearrowright "2";"12"};
\endxy
\]
To show the uniqueness, let
\[ U_C\overset{u_C}{\longrightarrow}C\overset{v_C}{\longrightarrow}V_C[1]\rightarrow U_C[1] \]
be a distinguished triangle with $U_C\in\mathcal{W},V_C\in\mathcal{V}$.
Suppose $c,c^{\prime}\in\mathcal{C}(M_f,C)$ satisfy
\[ \underline{c}\circ\underline{m_f}=\underline{g}=\underline{c^{\prime}}\circ\underline{m_f}. \]
By Proposition \ref{PropVanishLift}, $\underline{v_C}\circ(\underline{c}-\underline{c^{\prime}})\circ\underline{m_f}=0$ means $v_C\circ(c-c^{\prime})\circ m_f=0$.
Thus $v_C\circ(c-c^{\prime})$ factors through $U_A$, and thus $v_C\circ(c-c^{\prime})=0$ follows from $\mathcal{C}(U_A,V_C[1])=0$.
Thus $c-c^{\prime}$ factors through $U_C\in\mathcal{W}$, which means $\underline{c}=\underline{c^{\prime}}$.
\[
\xy
(-16,7)*+{B}="0";
(0,7)*+{M_f}="2";
(16,7)*+{U_A}="4";
(-16,-7)*+{U_C}="10";
(0,-7)*+{C}="12";
(16,-7)*+{V_C[1]}="14";
{\ar^{m_f} "0";"2"};
{\ar "2";"4"};
{\ar "0";"10"};
{\ar^{0} "4";"14"};
{\ar|*+{_{c-c^{\prime}}} "2";"12"};
{\ar_{u_C} "10";"12"};
{\ar_{v_C} "12";"14"};
{\ar@{}|\circlearrowright "0";"12"};
{\ar@{}|\circlearrowright "2";"14"};
\endxy
\]
\end{proof}

\begin{cor}\label{CorCoKerExist}
In $\underline{\mathcal{H}}$, any morphism has a cokernel and a kernel.
\end{cor}
\begin{proof}
We only show the construction of the cokernel.
For any $A,B\in\mathcal{H}$ and any $f\in\mathcal{C}(A,B)$, define $m_f\colon B\rightarrow M_f$ as in Lemma \ref{LemConstofM}.
Since $A,B\in\mathcal{C}^-$, it follows
\[ \underline{m_f}\circ\underline{f}=0,\quad M_f\in\mathcal{C}^- \]
by Lemma \ref{LemConstofM}.
By Proposition \ref{PropReflection}, there exists $z_M\colon M_f\rightarrow Z_M$ which gives a reflection $\underline{z_M}\colon M_f\rightarrow Z_M$ of $M_f$ along $\underline{\mathcal{C}}^+\hookrightarrow\underline{\mathcal{C}}$.
By Lemma \ref{LemHeart}, $Z_M$ satisfies $Z_M\in\mathcal{H}$.

We claim that $\underline{z_M}\circ\underline{m_f}\colon B\rightarrow Z_M$ is the cokernel of $\underline{f}$. Let $S$ be any object in $\mathcal{H}$, and let $s\in\mathcal{C}(B,S)$ be any morphism satisfying $\underline{s}\circ\underline{f}=0$.
\[
\xy
(-18,0)*+{A}="2";
(0,0)*+{B}="6";
(-2,10)*+{}="7";
(14,10)*+{S}="8";
(9,-7)*+{M_f}="4";
(18,-14)*+{Z_M}="10";
{\ar "6";"4"};
{\ar_{\underline{f}} "2";"6"};
{\ar_{\underline{m_f}} "6";"4"};
{\ar_{\underline{z_M}} "4";"10"};
{\ar_{\underline{s}} "6";"8"};
{\ar@{}|{\circlearrowright} "6";"7"};
{\ar@/^0.80pc/^{0} "2";"8"};
\endxy
\]
It suffices to show that there uniquely exists $\underline{t}\in\underline{\mathcal{H}}(Z_M,S)$ such that
\[ \underline{t}\circ\underline{z_M}\circ\underline{m_f}=\underline{s}. \]
This follows immediately from Proposition \ref{PropReflection} and Proposition \ref{PropUnivofM}.
\end{proof}

\section{Abelianess of the heart}

In this section, as the main theorem, we show that the heart $\underline{\mathcal{H}}$ becomes an abelian category, for any cotorsion pair $(\mathcal{U},\mathcal{V})$. Although propositions and lemmas in this section could be applied for objects in $\mathcal{C}^+$ or $\mathcal{C}^-$ (with certain modifications of the statement), we mainly consider objects in $\mathcal{H}$.

\begin{prop}\label{PropEpim}
Let $B,C\in\mathcal{H}$, and let
\[ A\overset{f}{\longrightarrow}B\overset{g}{\longrightarrow}C\overset{h}{\longrightarrow}A[1] \]
be a distinguished triangle in $\mathcal{C}$.
Let $m_g\colon C\rightarrow M_g$ be as in Lemma \ref{LemConstofM}.
Then the following are equivalent.
\begin{enumerate}
\item $\underline{g}$ is epimorphic in $\underline{\mathcal{H}}$.
\item $\underline{g}$ is epimorphic in $\underline{\mathcal{C}}^+$.
\item $M_g$ satisfies $\underline{\mathcal{C}}(M_g, \underline{\mathcal{C}}^+)=0$.
\item $M_g$ satisfies $\mathcal{C}(M_g, \mathcal{V}[1])=0$.
\end{enumerate}
The dual statement also holds for monomorphisms.
\end{prop}
\begin{proof}
First, we show the equivalence of {\rm (1)} and {\rm (2)}.
Obviously, {\rm (2)} implies {\rm (1)}.
To show the converse, let $S$ be any object in $\mathcal{C}^+$, and $\underline{s}\in\underline{\mathcal{C}}(C,S)$ be any morphism.
Let $k_S\colon K_S\rightarrow S$ be the morphism defined in Remark \ref{RemCoReflect}, which gives a coreflection of $S$ along $\underline{\mathcal{C}}^-\hookrightarrow\underline{\mathcal{C}}$. By Lemma \ref{LemHeart}, $K_S\in\mathcal{H}$.

By Remark \ref{RemCoReflect}, there exists $\underline{j}\in\underline{\mathcal{C}}^-(C,K_S)$ such that $\underline{k_S}\circ\underline{j}=\underline{s}$,
\[
\xy
(-8,0)*+{B}="0";
(10,0)*+{C}="2";
(30,0)*+{S}="4";
(20,12)*+{K_S}="6";
(20,-2)*+{}="10";
{\ar^{\underline{g}} "0";"2"};
{\ar_{\underline{s}} "2";"4"};
{\ar^{\underline{j}} "2";"6"};
{\ar^{\underline{k_S}} "6";"4"};
{\ar@{}|\circlearrowright "6";"10"};
\endxy
\]
and we have
\begin{eqnarray*}
\underline{j}=0 &\Longleftrightarrow& \underline{k_S}\circ\underline{j}=0\\
\underline{j}\circ\underline{g}=0 &\Longleftrightarrow& \underline{k_S}\circ(\underline{j}\circ\underline{g})=0.
\end{eqnarray*}
By {\rm (1)}, we have
\[ \underline{j}=0\ \ \Longleftrightarrow\ \ \underline{j}\circ\underline{g}=0. \]
Thus $\underline{s}=0$ if and only if $\underline{s}\circ\underline{g}=0$, i.e., $\underline{g}$ is epimorphic in $\underline{\mathcal{C}}^+$.

Second, we show that {\rm (2)} implies {\rm (3)}.
By Proposition \ref{PropUnivofM}, for any $S\in\underline{\mathcal{C}}^+$ we have an isomorphism
\[ -\circ\underline{m_g}\colon\underline{\mathcal{C}}(M_g,S)\overset{\cong}{\longrightarrow}\{\underline{s}\in\underline{\mathcal{C}}(C,S)\mid\underline{s}\circ\underline{g}=0\}. \]

Third, we show that {\rm (3)} implies {\rm (4)}.
This immediately follows from Proposition \ref{PropVanishLift}, since $\mathcal{V}[1]\subseteq\mathcal{C}^+$.

Finally, we show that {\rm (4)} implies {\rm (3)}.
Suppose $\mathcal{C}(M_g,\mathcal{V}[1])=0.$ For any $S\in\mathcal{C}^+$, let
\[ U_S\overset{u_S}{\longrightarrow}S\overset{v_S}{\longrightarrow}V_S[1]\rightarrow U_S[1] \]
be a distinguished triangle satisfying $U_S\in\mathcal{W},V_S\in\mathcal{V}$.
By $\mathcal{C}(M_g,\mathcal{V}[1])=0$, any morphism $s\in\mathcal{C}(M_g,S)$ factors through $U_S\in\mathcal{W}$, which means $\underline{s}=0$.
\end{proof}

\begin{lem}\label{LemSection}
Let $B,C\in\mathcal{H}$, and let
\[ A\overset{f}{\longrightarrow}B\overset{g}{\longrightarrow}C\overset{h}{\longrightarrow}A[1] \]
be a distinguished triangle in $\mathcal{C}$.
Take a diagram
\[
\xy
(-10,14)*+{U_A[-1]}="0";
(-10,0)*+{A}="2";
(-10,-14)*+{V_A}="4";
(10,14)*+{U_B[-1]}="10";
(10,0)*+{B}="12";
(10,-14)*+{V_B}="14";
(28,0)*+{C}="16";
{\ar^{f_U} "0";"10"};
{\ar^{f} "2";"12"};
{\ar_{f_V} "4";"14"};
{\ar_{u_A} "0";"2"};
{\ar_{v_A} "2";"4"};
{\ar^{u_B} "10";"12"};
{\ar^{v_B} "12";"14"};
{\ar_{g} "12";"16"};
{\ar@{}|\circlearrowright "0";"12"};
{\ar@{}|\circlearrowright "2";"14"};
\endxy
\]
where
\begin{eqnarray*}
U_A[-1]\overset{u_A}{\longrightarrow} A\overset{v_A}{\longrightarrow} V_A\rightarrow U_A\\
U_B[-1]\overset{u_B}{\longrightarrow} B\overset{v_B}{\longrightarrow} V_B\rightarrow U_B
\end{eqnarray*}
are distinguished triangles satisfying $U_A,U_B\in\mathcal{U}, V_A\in\mathcal{V}$ and $V_B\in\mathcal{W}$.

If $\underline{g}$ is an epimorphism in $\underline{H}$, then there exists a morphism $s\in\mathcal{C}(V_B,V_A)$ such that
\[ s\circ v_B\circ f=v_A. \]
\end{lem}
\begin{proof}
Let $C\overset{m_g}{\longrightarrow}M_g$ be as in Lemma \ref{LemConstofM}.
Remark $h$ factors $m_g$.
\[
\xy
(10,0)*+{C}="2";
(30,0)*+{A[1]}="4";
(20,-12)*+{M_g}="6";
(20,2)*+{}="10";
{\ar^{h} "2";"4"};
{\ar_{m_g} "2";"6"};
{\ar^{{}^{\exists}} "6";"4"};
{\ar@{}|\circlearrowright "6";"10"};
\endxy
\]
By Proposition \ref{PropEpim}, $\mathcal{C}(M_g[-1],V_A)\cong\mathcal{C}(M_g,V_A[1])=0$, and thus $v_A\circ h[-1]=0$.
\[
\xy
(-9,8)*+{C[-1]}="0";
(-9,-8)*+{M_g[-1]}="2";
(9,8)*+{A}="4";
(9,-8)*+{V_A}="6";
(2,-3)*+{}="8";
(-2,3)*+{}="10";
{\ar^{h[-1]} "0";"4"};
{\ar_{m_g[-1]} "0";"2"};
{\ar^{v_A} "4";"6"};
{\ar_{0} "2";"6"};
{\ar "2";"4"};
{\ar@{}|\circlearrowright "0";"8"};
{\ar@{}|\circlearrowright "6";"10"};
\endxy
\]
Thus there exists $q\in\mathcal{C}(B,V_A)$ such that $q\circ f=v_A$.
\[
\xy
(-30,0)*+{C[-1]}="0";
(-10,0)*+{A}="2";
(-10,-14)*+{V_A}="4";
(8,12)*+{U_B[-1]}="10";
(8,0)*+{B}="12";
(8,-14)*+{V_B}="14";
(2,-9)*+{}="15";
(-24,-9)*+{}="17";
{\ar^{h[-1]} "0";"2"};
{\ar^{f} "2";"12"};
{\ar_{0} "0";"4"};
{\ar|*+{_{v_A}} "2";"4"};
{\ar_{f_V} "4";"14"};
{\ar^{q} "12";"4"};
{\ar^{u_B} "10";"12"};
{\ar^{v_B} "12";"14"};
{\ar@{}|\circlearrowright "2";"15"};
{\ar@{}|\circlearrowright "2";"17"};
\endxy
\]

Since $\mathcal{C}(U_B[-1],V_A)=0$, we have $q\circ u_B=0$, and thus there exists $s\in\mathcal{C}(V_B,V_A)$ such that $s\circ v_B=q$.
Thus we obtain $s\circ v_B\circ f=q\circ f=v_A$.
\end{proof}

\begin{lem}\label{LemAinC-}
In the notation of Lemma \ref{LemSection}, if $\underline{g}$ is epimorphic in $\underline{\mathcal{H}}$, then $A$ belongs to $\mathcal{C}^-$.
\end{lem}
\begin{proof}
We use the notation of Lemma \ref{LemSection}.
By the same lemma, there exists $s\in\mathcal{C}(V_B,V_A)$ such that $s\circ v_B\circ f=v_A$.

It suffices to show $V_A\in\mathcal{W}$.
By assumption $V_B\in\mathcal{W}$ and thus for any $V^{\dag}\in\mathcal{V}$ and any $v^{\dag}\in\mathcal{C}(V_A,V^{\dag}[1])$, we have $v^{\dag}\circ s=0$.
Thus $v^{\dag}\circ v_A=v^{\dag}\circ s\circ v_B\circ f=0$, and $v^{\dag}$ factors through $U_A$, which implies $v^{\dag}=0$.
\[
\xy
(-10,12)*+{A}="0";
(10,12)*+{B}="1";
(-10,0)*+{V_A}="2";
(-10,-16)*+{U_A}="4";
(10,0)*+{V_B}="6";
(5.5,-12.5)*+{V^{\dag}[1]}="8";
(-5,-10)*+{}="10";
(-1,-17)*+{}="12";
{\ar_{v_A} "0";"2"};
{\ar^{f} "0";"1"};
{\ar^{v_B} "1";"6"};
{\ar "2";"4"};
{\ar_{s} "6";"2"};
{\ar|{v^{\dag}} "2";"8"};
{\ar^{0} "6";"8"};
{\ar@{-->}_{0} "4";"8"};
{\ar@{}|\circlearrowright "6";"10"};
{\ar@{}|\circlearrowright "0";"6"};
{\ar@{}|\circlearrowright "2";"12"};
\endxy
\]
\end{proof}

\begin{thm}\label{ThmHeartAbel}
For any cotorsion pair $(\mathcal{U},\mathcal{V})$ in $\mathcal{C}$, its heart
\[ \underline{\mathcal{H}}=(\mathcal{C}^+\cap\mathcal{C}^-)/(\mathcal{U}\cap\mathcal{V}) \]
is an abelian category.
\end{thm}
\begin{proof}
Since $\underline{\mathcal{H}}$ is an additive category with (co-)kernels (Corollary \ref{CorCoKerExist}), it remains to show the following:
\begin{enumerate}
\item If $\underline{g}$ is epimorphic in $\underline{\mathcal{H}}$, then $\underline{g}$ is a cokernel of some morphism in $\underline{\mathcal{H}}$.
\item If $\underline{g}$ is monomorphic in $\underline{\mathcal{H}}$, then $\underline{g}$ is a kernel of some morphism in $\underline{\mathcal{H}}$.
\end{enumerate}
Since {\rm (2)} can be shown dually, we only show {\rm (1)}.

Complete $g$ into a distinguished triangle
\[ A\overset{f}{\longrightarrow}B\overset{g}{\longrightarrow}C\overset{h}{\longrightarrow}A[1] \]
in $\mathcal{C}$.
By Lemma \ref{LemAinC-}, we have $A\in\mathcal{C}^-$.
By Proposition \ref{PropReflection}, there exists
\[ z_A\colon A\rightarrow Z_A \]
which gives a reflection $\underline{z_A}:A\rightarrow Z_A$ of $A$ along $\underline{\mathcal{C}}^+\hookrightarrow\underline{\mathcal{C}}$.
Moreover by Lemma \ref{LemHeart}, $Z_A$ belongs to $\mathcal{H}$.
Since $B$ belongs to $\mathcal{C}^+$, there exists $b\in\mathcal{C}(Z_A,B)$ such that $b\circ z_A=f$.
\[
\xy
(10,0)*+{A}="2";
(30,0)*+{B}="4";
(20,-12)*+{Z_A}="6";
(20,2)*+{}="10";
{\ar^{f} "2";"4"};
{\ar_{z_A} "2";"6"};
{\ar_{b} "6";"4"};
{\ar@{}|\circlearrowright "6";"10"};
\endxy
\]
We claim that $\underline{g}=\mathrm{cok}(\underline{b})$.

Let $S$ be any object in $\mathcal{H}$, and $s\in\mathcal{C}(B,S)$ be any morphism.
By Proposition \ref{PropReflection}, $\underline{s}\circ\underline{f}=0$ if and only if $\underline{s}\circ\underline{b}=0$.
In particular, $\underline{g}\circ\underline{b}=0$.

So it suffices to show for any $s$ satisfying $\underline{s}\circ\underline{f}=0$, there uniquely exists $\underline{c}\in\underline{\mathcal{C}}(C,S)$ such that $\underline{c}\circ\underline{g}=\underline{s}$.

Uniqueness immediately follows from the fact that $\underline{g}$ is epimorphic. So it remains to show the existence of $\underline{c}$.
Since $A\in\mathcal{C}^-$, there exists a distinguished triangle
\[ U_A[-1]\overset{u_A}{\longrightarrow}A\overset{v_A}{\longrightarrow}V_A\rightarrow U_A \]
with $U_A\in\mathcal{U},V_A\in\mathcal{W}$.
By assumption we have $\underline{s}\circ\underline{f}\circ\underline{u_A}=0$, and this means $s\circ f\circ u_A=0$ by Proposition \ref{PropVanishLift}.
So $s\circ f$ factors through $V_A$, and we obtain a morphism of triangles
\[
\xy
(-24,8)*+{A}="0";
(-8,8)*+{B}="2";
(8,8)*+{C}="4";
(24,8)*+{A[1]}="6";
(-24,-8)*+{V_A}="10";
(-8,-8)*+{S}="12";
(8,-8)*+{T}="14";
(24,-8)*+{V_A[1]}="16";
{\ar_{v_A} "0";"10"};
{\ar^{s} "2";"12"};
{\ar^{t} "4";"14"};
{\ar^{-v_A[1]} "6";"16"};
{\ar^{f} "0";"2"};
{\ar_{i} "10";"12"};
{\ar^{g} "2";"4"};
{\ar_{j} "12";"14"};
{\ar^{h} "4";"6"};
{\ar "14";"16"};
{\ar@{}|\circlearrowright "0";"12"};
{\ar@{}|\circlearrowright "2";"14"};
{\ar@{}|\circlearrowright "4";"16"};
\endxy
\]
where $V_A\overset{i}{\longrightarrow}S\overset{j}{\longrightarrow}T\rightarrow V_A[1]$ is a distinguished triangle.
By the same argument as in the proof of Lemma \ref{LemSection} (since $\mathcal{C}(M_g,V_A[1]=0)$), we have $v_A[1]\circ h=0$.
Thus $t$ factors through $S$, namely there exists $c\in\mathcal{C}(C,S)$ such that $j\circ c=t$.
Since $j\circ(s-c\circ g)=j\circ s-t\circ g=0$, there exists $s^{\prime}\in\mathcal{C}(B,V_A)$ such that
\begin{eqnarray*}
i\circ s^{\prime}=s-c\circ g. \\
\xy
(-26,-12)*+{V_A}="0";
(-10,0)*+{B}="2";
(-10,-12)*+{S}="4";
(6,0)*+{C}="12";
(6,-12)*+{T}="14";
(-6,-4)*+{}="15";
(-24,9)*+{}="17";
{\ar_{s^{\prime}} "2";"0"};
{\ar^{g} "2";"12"};
{\ar_{i} "0";"4"};
{\ar_{s} "2";"4"};
{\ar_{j} "4";"14"};
{\ar_{c} "12";"4"};
{\ar^{t} "12";"14"};
{\ar@{}|\circlearrowright "14";"15"};
\endxy
\end{eqnarray*}
Since $V_A\in\mathcal{W}$, this means $\underline{s}=\underline{c}\circ\underline{g}$.
\end{proof}

\section{Existence of enough projectives/injectives}

\begin{lem}\label{ProjLem1}
For any cotorsion pair $(\mathcal{U},\mathcal{V})$, the following are equivalent.
\begin{enumerate}
\item $\mathcal{U}\subseteq\mathcal{V}$.
\item $\mathcal{W}=\mathcal{U}$.
\item $\mathcal{C}^+=\mathcal{C}$.
\end{enumerate}
\end{lem}
\begin{proof}
Left to the reader.
\end{proof}

\begin{cor}\label{ProjCor2}
If a cotorsion pair $(\mathcal{U},\mathcal{V})$ satisfies $\mathcal{U}\subseteq\mathcal{V}$, then we have $\mathcal{U}[-1]\subseteq\mathcal{H}$.
\end{cor}
\begin{proof}
This immediately follows from $\mathcal{U}[-1]\subseteq\mathcal{C}^-$ (Remark \ref{RemRem}) and $\mathcal{C}=\mathcal{C}^+$ (Lemma \ref{ProjLem1}).
\end{proof}

\begin{prop}\label{ProjProp3}
Let $(\mathcal{U},\mathcal{V})$ be a cotorsion pair satisfying $\mathcal{U}\subseteq\mathcal{V}$. If an object $P\in\mathcal{H}$ lies in the image of $\mathcal{U}[-1]$ \ $($i.e. $P\in\mathcal{U}[-1]/\mathcal{W}$ $)$, then $P$ is projective in $\underline{\mathcal{H}}$.
\end{prop}
\begin{proof}
Let $B$ and $C$ be any objects in $\mathcal{H}$ and let $p\in\mathcal{C}(P,C)$ be any morphism.

Let $g\in\mathcal{C}(B,C)$ be any morphism which is epimorphic in $\underline{\mathcal{H}}$, and take a distinguished triangle
\[ A\overset{f}{\longrightarrow}B\overset{g}{\longrightarrow}C\overset{h}{\longrightarrow}A[1]. \]

By Proposition \ref{PropEpim} and Lemma \ref{ProjLem1}, $\underline{g}$ is epimorphic in $\underline{\mathcal{C}}$.
By Fact \ref{FactFact}, this is equivalent to $\underline{h}=0$.
Thus $h$ factors through some $W\in\mathcal{W}$.
Since $\mathcal{C}(P,W)\subseteq\mathrm{Ext}^1(\mathcal{U},\mathcal{V})=0$, we have $h\circ p=0$, and $p$ factors through $B$ as desired.
\end{proof}

\begin{cor}\label{ProjCor4}
If a cotorsion pair $(\mathcal{U},\mathcal{V})$ satisfies $\mathcal{U}\subseteq\mathcal{V}$, then its heart $\underline{\mathcal{H}}$ has enough projectives.
\end{cor}
\begin{proof}
By definition, for any $C\in\mathcal{H}$, there exists a distinguished triangle
\[ U[-1]\overset{u}{\longrightarrow}C\overset{v}{\longrightarrow}V\rightarrow U \]
with $U\in\mathcal{U}, V\in\mathcal{V}$.
Since $\underline{v}=0$ by $\underline{\mathcal{C}}(C,V)=0$ (Lemma \ref{LemDisApp}), $\underline{u}$ is epimorphic in $\underline{\mathcal{C}}$, and thus in $\underline{\mathcal{H}}$.
Thus Corollary \ref{ProjCor4} follows from Corollary \ref{ProjCor2} and Proposition \ref{ProjProp3}.
\end{proof}

Dually, we have the following.
\begin{cor}\label{ProjCor5}
If a cotorsion pair $(\mathcal{U},\mathcal{V})$ satisfies $\mathcal{V}\subseteq\mathcal{U}$, then its heart $\underline{\mathcal{H}}$ has enough injectives.
\end{cor}

\section{Example arising from an $n$-cluster tilting subcategory}

\begin{dfn}\label{ClusDef1}(\S\S 5.1 in \cite{K-R})
A full additive thick subcategory $\mathcal{T}\subseteq\mathcal{C}$ is an $n$-{\it cluster tilting subcategory} if it satisfies the following.
\begin{enumerate}
\item $\mathcal{T}$ is functorially finite.
\item An object $C\in\mathcal{C}$ belongs to $\mathcal{T}$ if and only if $\mathrm{Ext}^{\ell}(C,\mathcal{T})=0\ \ (0<{}^{\forall}\ell<n)$
\item An object $C\in\mathcal{C}$ belongs to $\mathcal{T}$ if and only if $\mathrm{Ext}^{\ell}(\mathcal{T},C)=0\ \ (0<{}^{\forall}\ell<n)$
\end{enumerate}
\end{dfn}

\begin{dfn}\label{ClusDef2}
Let $\mathcal{T}$ be an $n+1$-cluster tilting subcategory. For any pair of integers $i$ and $j$ satisfying $i\le j$, put
\begin{eqnarray*}
\mathcal{T}_{[i,j]}&:=&\{ C\in\mathcal{C}\mid\mathrm{Ext}^{\ell}(C,\mathcal{T})=0\ \ (j+1\le {}^{\forall}\ell\le n) \},\\
\mathcal{T}^{[i,j]}&:=&\{ C\in\mathcal{C}\mid\mathrm{Ext}^{\ell}(\mathcal{T},C)=0\ \ (-j\le {}^{\forall}\ell\le -i) \}.
\end{eqnarray*}
In particular, if $j-i\ge n$, then $\mathcal{T}_{[i,j]}=\mathcal{T}_{[i,j]}=\mathcal{C}$.
\end{dfn}

\begin{ex}\label{ClusEx3}
Let $\mathcal{T}$ be an $n+1$-cluster tilting subcategory. For any integer $\varpi$ satisfying $0\le\varpi\le n-1$, put
\begin{eqnarray*}
\mathcal{T}_{\varpi}&:=&\mathcal{T}_{[0,\varpi]},\\
\mathcal{T}^{\varpi}&:=&\mathcal{T}^{[\varpi,n-1]}.
\end{eqnarray*}
Then it can be shown that $(\mathcal{T}_{\varpi},\mathcal{T}^{\varpi})$ is a cotorsion pair (Theorem 3.1 in \cite{I-Y}).
We can calculate its heart $\underline{\mathcal{H}}_{\varpi}$ as follows.
\begin{eqnarray*}
\mathcal{C}^+=\mathcal{T}_{[\varpi,n]},\ \ \mathcal{C}^-=\mathcal{T}^{[-1,\varpi]},\ \ \mathcal{W}=\mathcal{T}[\varpi],\\
\underline{\mathcal{H}}_{\varpi}=(\mathcal{T}_{[\varpi,n]}\cap\mathcal{T}^{[-1,\varpi]})/\mathcal{T}[\varpi].
\end{eqnarray*}
Thus $(\mathcal{T}_{[\varpi,n]}\cap\mathcal{T}^{[-1,\varpi]})/\mathcal{T}[\varpi]$ is an abelian category for each $\varpi\in [0,n-1]$, as shown in Corollary 6.4 in \cite{I-Y}.
If $n=1$ and $\varpi=0$, this is nothing other than the case of a cluster tilting subcategory.

Moreover in Corollary 6,4 in \cite{I-Y}, it was shown that these $(\mathcal{T}_{[\varpi,n]}\cap\mathcal{T}^{[-1,\varpi]})/\mathcal{T}[\varpi]$ are mutually equivalent for $\varpi\in [0,n-1]$. We abbreviate these equivalent abelian categories by $\underline{\mathcal{H}}$.

Since $(\mathcal{T}_0,\mathcal{T}^0)$ satisfies $\mathcal{T}_0=\mathcal{T}\subseteq\mathcal{T}^0$, we see that $\underline{\mathcal{H}}$ has enough projectives by Corollary \ref{ProjCor4}. Dually, since $(\mathcal{T}_{n-1},\mathcal{T}^{n-1})$ satisfies $\mathcal{T}^{n-1}=\mathcal{T}[n-1]\subseteq\mathcal{T}_{n-1}$, we see $\underline{\mathcal{H}}$ has enough injectives by Corollary \ref{ProjCor5}.
\end{ex}

\end{document}